\documentclass[a4,11pt]{article}
\usepackage[subrefformat=parens]{subcaption}
\usepackage[dvips]{graphicx}
\usepackage{amssymb}
\usepackage{amstext}
\usepackage{amsmath}
\usepackage{amscd}
\usepackage{amsthm}
\usepackage{amsfonts}
\usepackage{enumerate}
\usepackage{graphicx}
\usepackage{latexsym}
\usepackage{mathrsfs}
\usepackage{mathtools}
\usepackage{cases}
\usepackage[svgnames]{xcolor}
\usepackage{verbatim}
\usepackage{bm}
\usepackage{tikz}
\newtheorem{thm}{Theorem}[section]

\newtheorem{lem}[thm]{Lemma}

\newtheorem{ass}[thm]{Assumption}
\theoremstyle{definition}

\newtheorem*{claim*}{Claim}
\theoremstyle{remark}

\numberwithin{equation}{section}

\title{\vspace{-3cm}\textbf{The direct and inverse scattering problem for the semilinear Schr\"{o}dinger equation}}

\author{Takashi FURUYA}

\date{}
\begin{document}
\maketitle
\begin{abstract}
We study the direct and inverse scattering problem for the semilinear Schr\"{o}dinger equation $\Delta u+a(x,u)+k^2u=0$ in $\mathbb{R}^d$. We show well-posedness in the direct problem for small solutions based on the Banach fixed point theorem, and the solution has the certain asymptotic behavior at infinity. We also show the inverse problem that the semilinear function $a(x,z)$ is uniquely determined from the scattering data. The idea is certain linearization that by using sources with several parameters we differentiate the nonlinear equation with respect to these parameter in order to get the linear one. (see \cite{A. Feizmohammadi and L. Oksanen, M. Salo1, M. Salo2}.)
\end{abstract}
\date{{\bf Key words}. inverse scattering problem, semilinear Schr\"{o}dinger equation, linearization, Herglotz wave function.}
%%%%%%%%%%%%%%%%%%%%%%%%%%%%%%%%%%%%%%%%%%%%%%%%%%%%%%%%%%%%%%%%%%%%%%%%%%%%%%%%%%%%%%%%%%%%%%%%%%%%%%%%%%%%%%%%%%%%%%%%%%%%%%%%%%%%%%%%%%%%%%%%%%%%%%%%%%%%%%%%%%%%%%%%%%%%%%%%%%%%%%%%%%%%%%%%%%%%%%%%%%%%%%%%%%%%%%%%%%%%%%%%%%%%%%%%%%%%%%%%%%%%%%%%%%%%%%%%%%%%%%%%%%%%%%%%%%%%%%%%%%%%%%%%%%%%%%%%%%%%%%%%
\section{Introduction}
In this paper, we study the direct and inverse scattering problem for the semilinear Schr\"{o}dinger equation 
\begin{equation}
\Delta u+a(x,u)+k^2u=0\  \mathrm{in} \ \mathbb{R}^d,\label{1.1}
\end{equation}
where $d \geq 2$, and $k>0$. Throughout this paper, we make the following assumptions for the semilinear function $a:\mathbb{R}^d \times \mathbb{C} \to \mathbb{C}$.
\begin{ass}
We assume that 
\begin{description}
\item[(i)] $a(x,0)=0$ for all $x  \in  \mathbb{R}^d$.
  
\item[(ii)] $a(x,z)$ is holomorphic at $z=0$ for each $x  \in  \mathbb{R}^d$, that is, there exists $\eta>0$ such that $a(x,z)=\sum_{l=1}^{\infty}\frac{\partial_{z}^{l}a(x,0)}{l!}z^l$ for $|z|<\eta$.
  
\item[(iii)] $\partial_{z}^{l}a(\cdot,0) \in L^{\infty}(\mathbb{R}^d)$ for all $l \geq 1$. Furthermore, there exists $c_0>0$ such that $\left\| \partial_{z}^{l}a(\cdot,0) \right\|_{L^{\infty}(\mathbb{R}^d)} \leq c_0^{l}$ for all $l\geq1$
  
\item[(iv)] There exists $R>0$ such that $\mathrm{supp}\partial_{z}^{l} a(\cdot,0) \subset B_R$ where $B_R \subset \mathbb{R}^d$ is a open ball with center 0 and radius $R>0$.
\end{description}
\end{ass}
The above assumptions include the standard type $q(x)u$ where $q \in L^{\infty}(\mathbb{R}^d)$ with compact support, and the power type $q(x)u^m$ where $m \in \mathbb{N}$. So far, the inverse problem for the power type in bounded domain via the Dirichlet-to-Neumann map has been studied in \cite{A. Feizmohammadi and L. Oksanen, M. Salo1}, and for more general case we refer to \cite{M. Salo2}, which also discusses partial data inverse boundary problem. 
\par
We consider the incident field $u^{in}_{g}$ as the {\it Herglotz wave function}
\begin{equation}
u^{in}_g(x):=\int_{\mathbb{S}^{d-1}}e^{-ikx\cdot \theta}g(\theta)ds(\theta), \ x \in \mathbb{R}^d, \ g \in L^{2}(\mathbb{S}^{d-1}), \label{1.2}
\end{equation}
which solves the free Schr\"{o}dinger equation $\Delta u^{in}_g+k^2u^{in}_g=0$ in $\mathbb{R}^d$. The scattered field $u^{sc}_g$ corresponding to the incident field $u^{in}_g$ is a solution of the following Schr\"{o}dinger equation perturbed by the semilinear function $a(x,z)$ 
\begin{equation}
\Delta u_g+a(x,u_g)+k^2u_g=0 \ \mathrm{in} \ \mathbb{R}^d,\label{1.3}
\end{equation}
where $u_g$ is total field that is of the form $u_g = u^{sc}_g +u^{in}_g$, and the scattered field $u^{sc}$ satisfies the {\it Sommerfeld radiation condition}
\begin{equation}
\lim_{r \to \infty} r^{\frac{d-1}{2}} \biggl( \frac{\partial u^{sc}}{\partial  r}-iku^{sc} \biggr)=0, \label{1.4}
\end{equation}
where $r=|x|$.
\par
Since support of the function $a(x,z)$ is compact, the direct scattering problem (\ref{1.3})--(\ref{1.4}) is equivalent to the following integral equation. (see e.g., the argument of Theorem 8.3 in \cite{D. Colton and R. Kress}.)
\begin{equation}
u_g(x)=u^{in}_g+\int_{\mathbb{R}^d}\Phi(x,y)a(y,u_g(y))dy, \ \ x \in \ \mathbb{R}^d,\label{1.5}
\end{equation}
where $\Phi(x,y)$ is the fundamental solution for $-\Delta - k^2$ in $\mathbb{R}^d$. In the following theorem, we will find a small solution $u^{sc}_g$ of (\ref{1.5}) for small $g \in L^{\infty}(\mathbb{R}^d)$.
\begin{thm}
We assume that $a(x,z)$ satisfies Assumption 1.1. Then, there exists $\delta_0 \in (0,1)$ such that for all $\delta \in (0, \delta_0)$ and $g \in L^{\infty}(\mathbb{R}^d)$ with $\left\| g \right\|_{L^{\infty}(\mathbb{R}^d)}<\delta^2$, there exists a unique solution $u^{sc}_{g} \in L^{\infty}(\mathbb{R}^d)$ with $\left\| u^{sc}_{g} \right\|_{L^{\infty}(\mathbb{R}^d)}\leq \delta$ such that 
\begin{equation}
u^{sc}_g(x)=\int_{\mathbb{R}^d}\Phi(x,y)a(y,u^{sc}_g(y)+u^{in}_g(y))dy, \ x \in \mathbb{R}^d.\label{1.6}
\end{equation}

\end{thm}
Theorem 1.2 is proved by the Banach fixed point theorem. By the same argument in Section 19 of \cite{G. Eskin}, the solution $u^{sc}_g$ of (\ref{1.6}) has the following asymptotic behavior
\begin{equation}
u^{sc}_g(x)=C_d\frac{\mathrm{e}^{ikr}}{r^{\frac{d-1}{2}}} u^{\infty}_g(\hat{x})+O\left(\frac{1}{r^{\frac{d+1}{2}}} \right), \ r:=|x| \to \infty, \ \ \hat{x}:=\frac{x}{|x|}. \label{1.7}
\end{equation}
where $C_d:=k^{\frac{d-3}{2}}e^{-i\frac{\pi}{4}(d-3)}/2^{\frac{d+1}{2}}\pi^{\frac{d-1}{2}}$. The function $u^{\infty}_g$ is called the {\it scattering amplitude}, which is of the form
\begin{equation}
u^{\infty}_g(\hat{x})=\int_{\mathbb{R}^d}e^{-ik\hat{x}\cdot y}a(y, u_g(y))dy, \ \hat{x} \in \mathbb{S}^{d-1}. \label{1.8}
\end{equation}
\par
Hence, we are now able to consider the inverse problem to determine the semilinear function $a(x,z)$ from the scattering data $u^{\infty}_{g}(\hat{x})$ for all $g \in L^{2}(\mathbb{S}^{d-1})$ with $\left\| g \right\|_{L^{2}(\mathbb{R}^d)}<\delta$ where $\delta>0$ is a sufficiently small. We will show the following theorem.
\begin{thm}
We assume that $a_j(x,z)$ satisfies Assumption 1.1. ($j=1,2$.) Let $u^{\infty}_{g,j}$ be the scattering amplitude for the following problem
\begin{equation}
\Delta u_{j,g}+a_j(x,u_{j,g})+k^2u_{j,g}=0 \ \mathrm{in} \ \mathbb{R}^d, \label{1.9}
\end{equation}
\begin{equation}
u_{j,g}=u^{sc}_{j,g}+u^{in}_{j,g}, \label{1.10}
\end{equation}
where $u^{sc}_{j,g}$ satisfies the Sommerfeld radiation (\ref{1.4}), and $u^{in}_g$ is given by (\ref{1.2}), and we assume that
\begin{equation}
u^{\infty}_{1,g}=u^{\infty}_{2,g}, \label{1.11}
\end{equation}
for any $g \in L^{2}(\mathbb{S}^{d-1})$ with $\left\| g \right\|_{L^{2}(\mathbb{R}^d)}<\delta$ where $\delta>0$ is sufficiently small. Then, we have
\begin{equation}
a_1(x,z)=a_2(x,z), \ x \in \mathbb{R}^d, \ |z| < \eta \label{1.12}
\end{equation}
\end{thm}
The idea of the proof is the linearization, which by using sources with several parameters we differentiate the nonlinear equation with respect to these parameter in order to get the linear one. (For such ideas, we refer to \cite{A. Feizmohammadi and L. Oksanen, M. Salo1, M. Salo2}.) The inverse scattering problems for non-linear Schr\"{o}dinger equation have been studied in different types of the non-linear potential $a(x,u)$ and in various ways. (See, e.g., \cite{M. Harju and V. Serov, V. Serov, V. Serov2, M. Watanabe, R. Weder1, R. Weder2}.) The feature of our works is to recover the whole nonlinearity $a(x,z)$ from the scattering data, that wavenumber $k>0$ is fixed and the incident wave is all of small Herglotz wave functions.
\par
This paper is organized as follows. In Section 2, we recall the Green function for the Helmholtz equation and its properties. We also prepare the several lemmas required in the forthcoming argument. In Section 3, we prove Theorem 1.2 based on the Banach fixed point theorem. In Section 4, we consider the special solution of (\ref{1.3})--(\ref{1.4}) corresponding to the incident field with several parameters in order to linearize problems. Finally in Section 5, we prove Theorem 1.3.  
%%%%%%%%%%%%%%%%%%%%%%%%%%%%%%%%%%%%%%%%%%%%%%%%%%%%%%%%%%%%%%%%%%%%%%%%%%%%%%%%%%%%%%%%%%%%%%%%%%%%%%%%%%%%%%%%%%%%%%%%%%%%
%%%%%%%%%%%%%%%%%%%%%%%%%%%%%%%%%%%%%%%%%%%%%%%%%%%%%%%%%%%%%%%%%%%%%%%%%%%%%%%%%%%%%%%%%%%%%%%%%%%%%%%%%%%%%%%%%%%%%%%%%%%%%%%%%%%%%%%%%%%%%%%%%%%%%%%%%%%%%%%%%%%%%%%%%%%%%%%%%%%%%%%%%%%%%%%%%%%%%%%%%%%%%%%%%%%%%%%%%%%%%%%%%%%%%%%%%%%%%%%%%%%%%%%%%%%%%%%%%%%%%%%%%%%%%%%%%%%%%%%%%%%%%%%%%%%%%%%%%%%%%%%%%%%%%%%%%%%%%%%%%%%%%%%%%%%%%%%%%%%%%%%%%%%%%%%%%%%%%%%%%%%%%%%%%%%%%%%%%%%%%%%%%%%%%%%%%%%%%%%%%%%%%%%%%%%%%%%%%%%%%%%%%%%%%%%%%%%%%%%%%%%%%%%%%%%%%%%%%%%%%%%%%%%%%%%%%%%%%%%%%%%%%%%%%%%%%%%%%%%%%%%%%%%%%%%%%%%%%%%%%%%%%%%%%%%%%%%%%%%%%%%%%%%%%%%%%%%%%%%%%%%%%%%%%%%%%%%%%%%%%%%%%%%%%%%%%%%%%%%%%%%%%%%%%%%%%%%%%%%%%%%%%%%%%%%%%%%%%%%%%%%%%%%%%%%%%%%%%%%%%%%%%%%%%%%%%%%%%%
\section{Preliminary}
First, we recall the Green functions for the Helmholtz equation and its properties. We denote the Green function for $-\Delta-k^2$ in $\mathbb{R}^d$ by $\Phi(x,y)$, that is, $\Phi(x,y)$ satisfies
\begin{equation}
(-\Delta-k^2)\Phi(x,y) = \delta(x-y), \label{2.1}
\end{equation}
for $x,y \in\mathbb{R}^d$, $x\neq y$. In the case of $d=2,3$, $\Phi(x,y)$ is of the form
\begin{eqnarray}
\Phi(x,y)=\left\{ \begin{array}{ll}
\frac{i}{4}H^{(1)}_{0}(k|x-y|) & \quad \mbox{for $x,y \in\mathbb{R}^2 $, $x\neq y$}  \\
\dfrac{e^{ik|x-y|}}{4\pi |x-y|} & \quad \mbox{for $x,y \in\mathbb{R}^3 $, $x\neq y$}
\end{array} \right.\label{2.2}
\end{eqnarray}
Let $q \in L^{\infty}(\mathbb{R}^d)$ with compact support. We denote the Green function for $-\Delta-k^2-q$ in $\mathbb{R}^d$ by $\Phi_q(x,y)$, that is, $\Phi_q(x,y)$ satisfies
\begin{equation}
(-\Delta-k^2-q)\Phi_q(x,y) = \delta(x-y). \label{2.3}
\end{equation}
for $x,y \in\mathbb{R}^d$, $x\neq y$. It is well known that for every fixed $y$, $\Phi(x,y)$ and $\Phi_q(x,y)$ satisfy the Sommerfeld radiation condition. 
\par
We also recall the asymptotics behavior of $\Phi(x,y)$ as $|x|\to \infty$. In Lemma 19.3 of \cite{G. Eskin}, $\Phi(x,y)$ has the following asymptotics behavior for every fixed $y$,
\begin{equation}
\Phi(x,y)=C_d\frac{\mathrm{e}^{ik|x-y|}}{|x-y|^{\frac{d-1}{2}}}+O\left(\frac{1}{|x-y|^{\frac{d+1}{2}}} \right), \ |x| \to \infty \label{2.4}
\end{equation}
and (see the proof of Theorem 19.5 in \cite{G. Eskin}) 
\begin{eqnarray}
\Phi(x,y)=\left\{ \begin{array}{ll}
O\left(\frac{1}{|x-y|^{d-2}} \right) & \quad \mbox{$d \geq 3$, $x\neq y$}  \\
O\left(\big| \mathrm{ln}|x-y| \big|\right) & \quad \mbox{$d=2$, $x\neq y$}
\end{array} \right.\label{2.5}
\end{eqnarray}
In Theorem 19.5 of \cite{G. Eskin}, for every $f \in L^{\infty}(\mathbb{R}^d)$ with compact support, $u(x)=\int_{\mathbb{R}^d}\Phi(x,y)f(y)dy$ is a unique radiating solution. (that is, $u$ satisfies the Sommerfeld radiation condition (\ref{1.4}).) Furthermore, $u$ has the following asymptotic behavior
\begin{equation}
u(x)=C_d\frac{\mathrm{e}^{ikr}}{r^{\frac{d-1}{2}}} u^{\infty}(\hat{x})+O\left(\frac{1}{r^{\frac{d+1}{2}}} \right), \ r=|x| \to \infty, \ \ \hat{x}:=\frac{x}{|x|}, \label{2.6}
\end{equation}
where the scattering amplitude $u^{\infty}$ is of the form
\begin{equation}
u^{\infty}(\hat{x})=\int_{\mathbb{R}^d}e^{-ik\hat{x}\cdot y}f(y)dy, \ \hat{x} \in \mathbb{S}^{d-1}. \label{2.7}
\end{equation}
\par
The following lemma is given by the same argument as in Lemma 10.4 of \cite{D. Colton and R. Kress} or Proposition 2.4 of \cite{M. Salo3}.
\begin{lem}
Let $q \in L^{\infty}(\mathbb{R}^d)$ with compact support in $B_R \subset \mathbb{R}^d$ where some $R>0$. We define the Helglotz operator $H:L^{2}(\mathbb{S}^{d-1}) \to L^{2}(B_R(0))$ by
\begin{equation}
Hg(x):=\int_{\mathbb{S}^{d-1}}e^{ikx\cdot \theta}g(\theta)d\theta, \ x \in B_R, \label{2.8}
\end{equation}
and define the operator $T_q:L^{2}(B_R) \to L^{2}(B_R)$ by $T_qf:=f+w\Big|_{B_R(0)}$ where $w$ is a radiating solution such that   
\begin{equation}
\Delta w+k^2w+qw=-qf \ \mathrm{in} \ \mathbb{R}^d. \label{2.9}
\end{equation}
We define the subspace $V$ of $L^2(B_R)$ by 
\begin{equation}
V:=\overline{ \left\{v\big|_{B_R} ; v \in L^{2}(B_{R+1}), \ \Delta v + k^2v +qv=0 \ \mathrm{in} \ \mathrm{B}_{R+1} \right\}}^{\left\| \cdot \right\|_{L^2(B_R)}}.\label{2.10}
\end{equation}
Then, the range of the operator $T_q H$ is dense in $V$ with respect to the norm $\left\| \cdot \right\|_{L^2(B_R)}$, that is,
\begin{equation}
\overline{ T_qH\left(L^2(\mathbb{S}^{d-1})\right) }^{\left\| \cdot \right\|_{L^2(B_R)}}=V.\label{2.11}
\end{equation}
\end{lem}
The following result is well known. For $d=2$ we refer to \cite{A. Bukhgeim}, and for $d\geq 3$ we refer to \cite{J. Sylvester and G. Uhlmann}.
\begin{lem}
Let $f, q_1, q_2 \in L^{\infty}(\mathbb{R}^d)$ with compact support in $B_R \subset \mathbb{R}^d$. We assume that
\begin{equation}
\int_{B_R}fv_1v_2dx=0,\label{2.24}
\end{equation}
for all $v_1, v_2 \in L^{2}(B_{R+1})$ with $\Delta v_j +k^2v_j+q_jv_j=0$ in $B_{R+1}$. ($j=1,2$.) Then, $f=0$ in $B_R$.
\end{lem}
%%%%%%%%%%%%%%%%%%%%%%%%%%%%%%%%%%%%%%%%%%%%%%%%%%%%%%%%%%%%%%%%%%%%%%%%%%%%%%%%%%%%%%%%%%%%%%%%%%%%%%%%%%%%%%%%%%%%%%%%%%%%%%%%%%%%%%%%%%%%%%%%%%%%%%%%%%%%%%%%%%%%%%%%%%%%%%%%%%%%%%%%%%%%%%%%%%%%%%%%%%%%%%%%%%%%%%%%%%%%%%%%%%%%%%%%%%%%%%%%%%%%%%%%%%%%%%%%%%%%%%%%%%%%%%%%%%%%%%%%%%%%%%%%%%%%%%%%%%%%%%%%%%%%%%
%%%%%%%%%%%%%%%%%%%%%%%%%%%%%%%%%%%%%%%%%%%%%%%%%%%%%%%%%%%%%%%%%%%%%%%%%%%%%%%%%%%%%%%%%%%%%%%%%%%%%%%%%%%%%%%%%%%%%%%%%%%%%%%%%%%%%%%%%%%%%%%%%%%%%%%%%%%%%%%%%%%%%%%%%%%%%%%%%%%%%%%%%%%%%%%%%%%%%%%%%%%%%%%%%%%%%%%%%%%%%%%%%%%%%%%%%%%%%%%%%%%%%%%%%%%%%%%%%%%%%%%%%%%%%%%%
\section{Proof of Theorem 1.2}
In Section 3, we will show Theorem 1.2 based on the Banach fixed point theorem. We denote the Herglotz wave function by
\begin{equation}
v_g(x):=\int_{\mathbb{S}^{d-1}}e^{-ikx\cdot \theta}g(\theta)ds(\theta), \ x \in \mathbb{R}^d, \ g \in L^{2}(\mathbb{S}^{d-1}). \label{3.0}
\end{equation}
Let $q:=\partial_{z}a(\cdot,0)$. We define the operator $T:L^{\infty}(\mathbb{R}^d) \to L^{\infty}(\mathbb{R}^d)$ by
\begin{eqnarray}
Tw(x)
&:=&\int_{\mathbb{R}^d}\Phi_q(x,y)\left[a\bigl(y,w(y)+v_g(y) \bigr)-q(y)w(y) \right]dy
\nonumber\\
&=&\int_{\mathbb{R}^d}\Phi_q(x,y)\left[\sum_{l\geq 2}\dfrac{\partial^{l}_{z}a(y,0)}{l!}\bigl(w(y)+v_g(y) \bigr)^l+q(y)v_g(y) \right]dy, \ x \in \mathbb{R}^d.\nonumber\\ \label{3.1}
\end{eqnarray}
Let $X_{\delta}:=\left\{u \in L^{\infty}(\mathbb{R}^d): \left\| u \right\|_{L^{\infty}(\mathbb{R}^d)} \leq \delta \right\}$. We remark that $L^{\infty}(\mathbb{R}^d)$ is a Banach space, and $X_{\delta}$ is closed subspace in $L^{\infty}(\mathbb{R}^d)$. To find an unique fixed point of $T$ in $X$, we will show that $T:X_{\delta} \to X_{\delta}$ and $T$ is a contraction. Let $w \in X_{\delta}$, and let $\delta \in (0,\delta_0)$, and let $\left\| g \right\|_{L^{\infty}(\mathbb{R}^d)}<\delta^2$. Later, we will choose a appropriate $\delta_0>0$. 
\par 
By $\left\| g \right\|_{L^{\infty}(\mathbb{R}^d)}<\delta^2$, we have
\begin{equation}
\left\|v_g \right\|_{L^{\infty}(\mathbb{R}^d)} \leq C \left\| g \right\|_{L^{\infty}(\mathbb{R}^d)} \leq C \delta^2
\end{equation}
where $C>0$ is constant only depending on $g$. By (iii) (iv) of Assumption 1.1, we have 
\begin{eqnarray}
|Tw(x)|
&\leq&\int_{B_R}|\Phi_q(x,y)|\left[\sum_{l\geq 2}\dfrac{c_0^l}{l!}\bigl(C_1 \delta \bigr)^l+C_1\delta^2 \right]dy
\nonumber\\
&\leq&C_2\delta^2  \left( \sum_{l\geq 0}\bigl(C_1 c_0 \delta \bigr)^l \right) \int_{B_R}|\Phi_q(x,y)|dy,\label{3.2}
\end{eqnarray}
where $C_j >0$ ($j=1,2$) is constant independent of $u$ and $\delta$, and so is $\left( \sum_{l\geq 0}\bigl(C_1 c_0 \delta \bigr)^l \right)$ when $\delta>0$ is sufficiently small. Furthermore, by the continuity of difference $\Phi(x,y)-\Phi_q(x,y)$ in $x$ and $y$ (see the proof of Theorem 31.6 in \cite{G. Eskin}), and the estimation (\ref{2.5}), we have for $x \in \mathbb{R}^d$
\begin{eqnarray}
\int_{B_R}|\Phi_q(x,y)|dy
&\leq&\int_{B_R}\bigl(|\Phi(x,y)|+|\Phi_q(x,y)-\Phi(x,y)| \bigr)dy
\nonumber\\
&\leq& \int_{B_R}\bigl(|\Phi(x,y)|+C_3 \bigr)dy
\leq C_4, \label{3.3}
\end{eqnarray}
which implies that $|Tw(x)| \leq C \delta^2$ where $C, C_j >0$ ($j=3,4$) is constant independent of $u$ and $\delta$. By choosing $\delta_0 \in (0, 1/C)$, we conclude that $\left\|Tw\right\|\leq \delta$, which means $Tw \in X_{\delta}$. 
\par
Let $w_1,w_2 \in X_{\delta}$. Since we have
\begin{eqnarray}
&&\bigl(w_1(y)+v_g(y) \bigr)^l-\bigl(w_2(y)+v_g(y) \bigr)^l
\nonumber\\
&=&\sum_{m=1}^{l}\frac{l!}{(l-m)!m!}\bigl(w^m_1(y)-w^m_2(y) \bigr)v_{g}^{l-m}(y)
\nonumber\\
&\leq&\sum_{m=1}^{l}\frac{l!}{(l-m)!m!}\left(\sum_{h=0}^{m-1}w^{m-1-h}_1(y)w^h_2(y) \right)\bigl( w_1(y)-w_2(y)\bigr)v_{g}^{l-m}(y), \nonumber\\ \label{3.4}
\end{eqnarray}
and $|w_j(x)|\leq \delta$, then 
\begin{eqnarray}
&&\left|Tw_1(x)-Tw_2(x)\right|
\nonumber\\
&=&\left| \int_{B_R}\Phi_q(x,y) \sum_{l\geq2}\dfrac{\partial^{l}_{z}a(y,0)}{l!} \left[\bigl(w_1(y)+v_g(y) \bigr)^l-\bigl(w_2(y)+v_g(y) \bigr)^l\right]dy \right|
\nonumber\\
&\leq& \left( \int_{B_R}\left|\Phi_q(x,y)\right|dy \right) \sum_{l\geq2}\dfrac{c_0^l}{l!} \sum_{m=1}^{l} \dfrac{l!}{(l-m)!m!} \left(\sum_{h=0}^{m-1}\delta^{m-1} \right) \left(C'_1\delta \right)^{l-m} \left\|w_1-w_2 \right\|_{L^{\infty}(\mathbb{R}^d)}
\nonumber\\
&\leq&C'_2 \sum_{l\geq2}\sum_{m=1}^{l} \frac{m}{(l-m)!m!}\left(c_0C'_1\delta \right)^{l-1} \left\|w_1-w_2 \right\|_{L^{\infty}(\mathbb{R}^d)}
\nonumber\\
&\leq&C'_2 \sum_{l\geq2}\left(\sum_{m=1}^{\infty} \frac{1}{(m-1)!} \right) \left(c_0C'_1\delta \right)^{l-1} \left\|w_1-w_2 \right\|_{L^{\infty}(\mathbb{R}^d)}
\nonumber\\
&\leq&C'_3\sum_{l\geq2}\left(c_0C'_1\delta \right)^{l-1}\left\|w_1-w_2 \right\|_{L^{\infty}(\mathbb{R}^d)}
\nonumber\\
&\leq&
C'_3 \left(\sum_{l\geq 0}\left(c_0C'_1\delta \right)^{l}\right)\delta \left\|w_1-w_2 \right\|_{L^{\infty}(\mathbb{R}^d)} 
\nonumber\\
&\leq&
C' \delta \left\|u_1-u_2 \right\|_{L^{\infty}(\mathbb{R}^d)},\ x \in \mathbb{R}^d. \label{3.5}
\end{eqnarray}
where $C',C'_j >0$ ($j=1,2,3$) is constant independent of $w_1,w_2$ and $\delta$. (We remark that $\left(\sum_{l\geq 0}\left(c_0C'_1\delta \right)^{l}\right)$ is also constant when $\delta>0$ is sufficiently small.) By choosing $\delta_0 \in (0, 1/C')$, we have $\left\|Tw_1-Tw_2 \right\|_{L^{\infty}(\mathbb{R}^d)}<\left\|w_1-w_2 \right\|_{L^{\infty}(\mathbb{R}^d)}$. Choosing sufficiently small $\delta_0 \in \left(0, \mathrm{min}(1/C,1/C') \right)$ we conclude that $T$ has a unique fixed point in $X_{\delta}$. 
\par
Let $w \in X_{\delta}$ be a unique fixed point, that is, $w$ satisfies
\begin{equation}
w(x)=\int_{\mathbb{R}^d}\Phi_q(x,y)\left[a\bigl(y,w(y)+v_g(y) \bigr)-q(y)w(y) \right]dy , \ x \in \mathbb{R}^d. \label{3.6}
\end{equation}
Since $\Phi_q(x,y)$ satisfy the Sommerfeld radiation condition (e.g., see Theorem 31.6 in \cite{G. Eskin}), $w$ is a radiating solution of $\Delta w+ a(x, w+v_g) + k^2w = 0$ in $\mathbb{R}^d$. By the same argument as in Theorem 8.3 of \cite{D. Colton and R. Kress}, this is equivalent to the integral equation
\begin{equation}
w(x)=\int_{\mathbb{R}^d}\Phi(x,y)a\bigl(y,w(y)+v_g(y) \bigr)dy , \ x \in \mathbb{R}^d, \label{3.7}
\end{equation}
which means (\ref{1.6}). Therefore, Theorem 1.2 has been shown.
%%%%%%%%%%%%%%%%%%%%%%%%%%%%%%%%%%%%%%%%%%%%%%%%%%%%%%%%%%%%%%%%%%%%%%%%%%%%%%%%%%%%%%%%%%%%%%%%%%%%%%%%%%%%%%%%%%%%%%%%%%%%%%%%%%%%%%%%%%%%%%%%%%%%%%%%%%%%%%%%%%%%%%%%%%%%%%%%%%%%%%%%%%%%%%%%%%%%%%%%%%%%%%%%%%%%%%%%%%%%%%%%%%%%%%%%%%%%%%%%%%%%%%%%%%%%%%%%%%%%%%%%%%%%
%%%%%%%%%%%%%%%%%%%%%%%%%%%%%%%%%%%%%%%%%%%%%%%%%%%%%%%%%%%%%%%%%%%%%%%%%%%%%%%%%%%%%%%%%%%%%%%%%%%%%%%%%%%%%%%%%%%%%%%%%%%%%%%%%%%%%%%%%%%%%%%%%%%%%%%%%%%%%%%%%%%%%%%%%%%%%%%%%%%%%%%%%%%%%%%%%%%%%%%%%%%%%%%%%%%%%%%%%%%%%%%%%%%%%%%%%%%%%%%%%%%%%%%%%%%%%%%%%%%%%%%%%%%%%%%%
\section{The special solution}
In Section 4, we consider the special solution of (\ref{1.3})--(\ref{1.4}) corresponding to the incident field with several parameters in order to linearize problems. Let $N \in \mathbb{N}$ be fixed and let $g_j \in L^2(\mathbb{S}^{d-1})$ be fixed ($j=1,2,...,N+1$). We set 
\begin{equation}
v_{\epsilon}:=\sum_{j=1}^{N+1}\epsilon_j \delta^2 v_{g_j}=v_{\left(\delta^2 \sum_{j=1}^{N+1}\epsilon_j g_j \right) }, \label{4.1}
\end{equation}
where $v_{g_j}$ is the Herglotz wave function defined by (\ref{1.2}), and $\epsilon_j \in (0,\delta)$. Later, we will choose a appropriate $\delta=\delta_{g_j, N}>0$. We remark that we can estimate that
\begin{equation}
\left\| v_{\epsilon} \right\|_{L^{\infty}(\mathbb{R}^d)} \leq C \delta^2 \sum_{j=1}^{N+1}\epsilon_j,\label{4.2}
\end{equation}
where $C>0$ is constant only depending on $g_j$. We denote by $\epsilon = (\epsilon_1, ..., \epsilon_{N+1}) \in \mathbb{R}^{N+1}$. We will find a small solution $u_\epsilon$ of (\ref{1.6}) that is of the form 
\begin{equation}
u_\epsilon = r_\epsilon + v_\epsilon. \label{4.3}
\end{equation} 
This problem is equivalent to 
\begin{equation}
r_\epsilon(x) = \int_{\mathbb{R}^d}\Phi_q(x,y)\left[a\bigl(y,r_\epsilon(y)+v_\epsilon(y) \bigr)-q(y)r_\epsilon(y) \right]dy, \ x \in \mathbb{R}^d,\label{4.4}
\end{equation} 
where $q:=\partial_{z}a(\cdot,0)$.
\par
We define the space for $\delta>0$ 
\begin{equation}
\tilde{X}_{\delta}:=\left\{r \in L^{\infty}(\mathbb{R}^d;C^{N+1}(0,\delta)^{N+1}) ;
\begin{array}{cc}
      \mathrm{ess.sup}_{x \in \mathbb{R}^d}|r(x,\epsilon)|\leq  \sum_{j=1}^{N+1}\epsilon_j, \\
      \left\| r \right\|_{L^{\infty}(\mathbb{R}^d;C^{N+1}(0,\delta)^{N+1})} \leq \delta,
\end{array}
\right\}, \label{4.5}
\end{equation}
where the norm $\left\| \cdot \right\|_{L^{\infty}(\mathbb{R}^d;C^{N+1}(0,\delta)^{N+1})}$ is defined by
\begin{equation}
\left\| r \right\|_{L^{\infty}(\mathbb{R}^d;C^{N+1}(0,\delta)^{N+1})}:=\sum_{|\alpha|\leq N+1}\mathrm{sup}_{\epsilon \in (0,\delta)^{N+1}}\mathrm{ess.sup}_{x \in \mathbb{R}^d}\left| \partial_{\epsilon}^{\alpha}r(x,\epsilon) \right|. \label{4.6}
\end{equation}
We remark that $L^{\infty}(\mathbb{R}^d;C^{N+1}(0,\delta)^{N+1})$ is a Banach space, and $\tilde{X}_{\delta}$ is closed subspace in $L^{\infty}(\mathbb{R}^d;C^{N+1}(0,\delta)^{N+1})$. We will show that following lemma in the same way of Theorem 1.2.
\begin{lem}
We assume that $a(x,z)$ satisfies Assumption 1.1. Then, there exists $\tilde{\delta}_0=\tilde{\delta}_{0, g_j, N} \in (0,1)$  such that for all $\delta \in (0, \tilde{\delta}_0)$ there exists an unique solution $r \in \tilde{X}_{\delta}$ such that
\begin{equation}
r(x,\epsilon)=\int_{\mathbb{R}^d}\Phi_q(x,y)\left[a\bigl(y, r(y,\epsilon)+v_{\epsilon}(y) \bigr)-q(y)r(y,\epsilon) \right]dy, \ x \in \mathbb{R}^d, \ \epsilon \in (0,\delta)^{N+1}.\label{4.7}
\end{equation}
\end{lem}
\begin{proof}
We define the operator $\tilde{T}$ from $L^{\infty}(\mathbb{R}^d;C^{N+1}(0,\delta)^{N+1})$ into itself by
\begin{eqnarray}
\tilde{T}r(x,\epsilon)
&:=&\int_{\mathbb{R}^d}\Phi_q(x,y)\left[a\bigl(y,r(y,\epsilon)+v_{\epsilon}(y) \bigr)-q(y)r(y,\epsilon) \right]dy
\nonumber\\
&=&\int_{\mathbb{R}^d}\Phi_q(x,y)\left[\sum_{l\geq 2}\dfrac{\partial^{l}_{z}a(y,0)}{l!}\bigl(r(y,\epsilon)+v_{\epsilon}(y) \bigr)^l+q(y)v_{\epsilon}(y) \right]dy
\nonumber\\
&=&\int_{\mathbb{R}^d}\Phi_q(x,y)\left[\sum_{l\geq 2}\dfrac{\partial^{l}_{z}a(y,0)}{l!}\sum_{m=0}^{l}\frac{l!}{(l-m)!m!}r^{l-m}(y,\epsilon)v^m_{\epsilon}(y) +q(y)v_{\epsilon}(y) \right]dy
\nonumber\\ \label{4.8}
\end{eqnarray}
Let $r \in \tilde{X}_{\delta}$. With (\ref{4.2}) we have
\begin{eqnarray}
&&\left|\tilde{T}r(x,\epsilon)\right|
\nonumber\\
&\leq&\left( \int_{B_R} \left|\Phi_q(x,y)\right| dy\right) \left[\sum_{l\geq 2}c_0^l \sum_{m=0}^{l} \frac{1}{m!}\left(\sum_{j=1}^{N+1}\epsilon_j \right)^{l-m}\left(\tilde{C}_1 \delta^2 \sum_{j=1}^{N+1}\epsilon_j \right)^{m} + \tilde{C}_1 \delta^2 \sum_{j=1}^{N+1}\epsilon_j \right]
\nonumber\\
&\leq& \tilde{C}_2 \left[\sum_{l\geq 2}c_0^l \left( \sum_{m=0}^{\infty} \frac{\tilde{C}_1^m}{m!} \right) \left(\sum_{j=1}^{N+1}\epsilon_j \right)^{l} + \tilde{C}_1 \delta^2 \sum_{j=1}^{N+1}\epsilon_j \right]
\nonumber\\
&\leq& \tilde{C}_3 \left(\sum_{j=1}^{N+1}\epsilon_j \right)^2 \sum_{l\geq 2}c_0^l \left(\sum_{j=1}^{N+1}\epsilon_j \right)^{l-2} + \tilde{C}_3 \delta \left(\sum_{j=1}^{N+1}\epsilon_j \right)  
\nonumber\\
&\leq& \tilde{C} \delta \left(\sum_{j=1}^{N+1}\epsilon_j \right),\label{4.9}
\end{eqnarray}
where $\tilde{C}, \tilde{C}_j >0$ ($j=1,2$) is constant independent of $r$, $\delta$, $\epsilon$ (but, depending on $g_j$ and $N$). Furthermore, we consider for $\alpha \in \mathbb{N}^{N+1}$ with $|\alpha|\leq N+1$
\begin{eqnarray}
&&\partial_{\epsilon}^{\alpha} \tilde{T}r(x,\epsilon)
\nonumber\\
&=&\int_{\mathbb{R}^d}\Phi_q(x,y)\partial_{\epsilon}^{\alpha} \left[\sum_{l\geq 2}\dfrac{\partial^{l}_{z}a(y,0)}{l!}\sum_{m=0}^{l}\frac{l!}{(l-m)!m!}r^{l-m}(y,\epsilon)v^m_{\epsilon}(y) +q(y)v_{\epsilon}(y) \right]dy.
\nonumber\\ \label{4.10}
\end{eqnarray}
Since $|\partial_{\epsilon_j}v_{\epsilon}(x)|\leq \tilde{C}'_1 \delta^2$ and $|\partial^{\alpha}_{\epsilon}r^{l-m}(x,\epsilon) v^m_{\epsilon}(x)|\leq \tilde{C}'_2  (l-m)! m!  \delta^{l-m} (\tilde{C}'_2\delta^2)^m$, we have
\begin{eqnarray}
\left| \partial_{\epsilon}^{\alpha} \tilde{T}r(x,\epsilon)\right|
&\leq&\left( \int_{B_R} \left|\Phi_q(x,y)\right| dy\right) \left[\sum_{l\geq 2}\dfrac{c_0^l}{l!} \sum_{m=0}^{l}\frac{l!m! (l-m)!}{(l-m)!m!} \delta^{l+m}(\tilde{C}'_2)^m +\tilde{C}'_3 \delta^2 \right]
\nonumber\\
&\leq& \tilde{C}'_4 \delta^2 \left( \sum_{l\geq2} (c_0\delta)^{(l-2)} \sum_{m=0}^{\infty}(\tilde{C}'_2 \delta)^m \right) + \tilde{C}'_4 \delta^2 \leq \tilde{C}'_5 \delta^2, \label{4.11}
\end{eqnarray}
where $\tilde{C}'_j >0$ ($j=3,4,5$) is also constant independent of $r$, $\delta$, $\epsilon$ (but depending on $\alpha$). Then, we have
\begin{equation}
\sum_{|\alpha|\leq N+1}\mathrm{sup}_{\epsilon \in (0,\delta)^{N+1}}\mathrm{ess.sup}_{x \in \mathbb{R}^d}\left| \partial_{\epsilon}^{\alpha}\tilde{T}r(x,\epsilon) \right| \leq \tilde{C}' \delta ^2,\label{4.12}
\end{equation}
where $\tilde{C}'$ is constant independent of $r$, $\delta$, $\epsilon$. (depending on $g_j$ and $N$.) By choosing $\tilde{\delta}_{0} \in \left(0, \mathrm{min}(1/\tilde{C},1/\tilde{C}') \right)$, we conclude that $\tilde{T}r \in \tilde{X}_{\delta}$.
\par
Let $r_1,r_2 \in \tilde{X}_{\delta}$. By similar argument in (\ref{3.4}) we have 
\begin{eqnarray}
&&\tilde{T}r_1(x,\epsilon)-\tilde{T}r_2(x,\epsilon)
\nonumber\\
&=& \int_{B_R}\Phi_q(x,y) \sum_{l\geq2}\dfrac{\partial^{l}_{z}a(y,0)}{l!} \left[\bigl(r_1(y,\epsilon)+v_{\epsilon}(y) \bigr)^l-\bigl(r_2(y, \epsilon)+v_{\epsilon}(y) \bigr)^l\right]dy 
\nonumber\\
&=& \int_{B_R}\Phi_q(x,y) \sum_{l\geq2}\dfrac{\partial^{l}_{z}a(y,0)}{l!}\sum_{m=1}^{l}\frac{l!}{(l-m)!m!}v^{l-m}_{\epsilon}(y)
\nonumber\\
&& \ \ \ \ \ \ \ \ \ \ \ \ \ \ \ \ \ \ \ \times 
\sum_{h=0}^{m-1} r^{m-1-h}_1(y,\epsilon)r^{h}_2(y, \epsilon)\left(r_1(y,\epsilon)-r_2(y,\epsilon) \right)dy.
\nonumber\\ \label{4.13}
\end{eqnarray}
Then, we have for $\alpha \in \mathbb{N}^{N+1}$ with $|\alpha|\leq N+1$
\begin{eqnarray}
&&\left| \partial_{\epsilon}^{\alpha} \left( \tilde{T}r_1(x)-\tilde{T}r_2(x) \right) \right|
\nonumber\\
&\leq& \int_{B_R} \left|\Phi_q(x,y)\right|  \sum_{\beta \leq \alpha} \frac{\alpha !}{ (\alpha - \beta)! \beta !} \sum_{l\geq 2}\dfrac{\left|\partial^{l}_z a(y,0)\right|}{l!} \sum_{m=1}^{l}\frac{l!}{(l-m)!m!}
\nonumber\\
&\times& \sum_{h=0}^{m-1} \left| \partial_{\epsilon}^{\beta}\left( v_{\epsilon}^{l-m}(y)r^{m-1-h}_1(y,\epsilon)r^{h}_2(y,\epsilon) \right) \right|  \left|\partial_{\epsilon}^{\alpha - \beta} \left( r_1(y,\epsilon)-r_2(y,\epsilon) \right)\right|dy.
\nonumber\\ \label{4.14}
\end{eqnarray}
Since 
\begin{equation}
\left| \partial_{\epsilon}^{\beta}\left( v_{\epsilon}^{l-m}(y)r^{m-1-h}_1(y,\epsilon)r^{h}_2(y,\epsilon) \right) \right| \leq \tilde{C}''_1 (l-m)!(m-1-h)!h! (\tilde{C}''_1 \delta^2)^{l-m}\delta^{m-1-h}\delta^{h}, \label{4.15}
\end{equation}
where $\tilde{C}''_1$ is constant independent of $r_1,r_2$ and $\delta$ (depending on $\beta$), we have that
\begin{eqnarray}
&&\left| \partial_{\epsilon}^{\alpha} \left( \tilde{T}r_1(x)-\tilde{T}r_2(x) \right) \right|
\nonumber\\
&\leq& \tilde{C}''_2 \left( \sum_{\beta \leq \alpha} \frac{\alpha !}{ (\alpha - \beta)! \beta !} \sum_{l\geq 2}\frac{c_0^l}{l!} \sum_{m=1}^{l} \sum_{h=0}^{m-1} \frac{l!(l-m)!(m-1-h)!h!}{(l-m)!m!} \delta^{2l-m-1} (\tilde{C}''_1)^{l-m} \right) \left\| r_1-r_2 \right\|
\nonumber\\
&\leq& \tilde{C}''_3 \delta \left( \sum_{l\geq2}(c_0\delta)^{l-2} \sum_{m=1}^{l} (\tilde{C}''_1 \delta)^{l-m}  \sum_{h=0}^{m-1} \frac{(m-1-h)!h!}{m!} \right) \left\| r_1-r_2 \right\|
\nonumber\\
&\leq& \tilde{C}''_4 \delta \left( \sum_{l\geq2}(c_0\delta)^{l-2} \sum_{p=0}^{\infty} (\tilde{C}''_1 \delta)^{p} \right) \left\| r_1-r_2 \right\| \leq \tilde{C}''_5 \delta \left\| r_1-r_2 \right\|_{L^{\infty}(\mathbb{R}^d;C^{N+1}(0,\delta)^{N+1})}, \label{4.16}
\end{eqnarray}
which implies that
\begin{equation}
\sum_{|\alpha|\leq N+1}\mathrm{sup}_{\epsilon \in (0,\delta)^{N+1}}\mathrm{ess.sup}_{x \in \mathbb{R}^d}\left| \partial_{\epsilon}^{\alpha} \left( \tilde{T}r_1(x,\epsilon)-\tilde{T}r_2(x,\epsilon) \right) \right| \leq \tilde{C}'' \delta \left\| r_1-r_2 \right\|, \label{4.17}
\end{equation}
where $\tilde{C}''_j,\tilde{C}''>0$ ($j=2,3,4$) is constant independent of $r_1,r_2$ and $\delta$. By choosing $\tilde{\delta}_0 \in \left( 0, \mathrm{min}(1/\tilde{C},1/\tilde{C}',1/\tilde{C}'')\right)$, we have $\left\|Tr_1-Tr_2 \right\|<\left\|r_1-r_2 \right\|$, which implies that $\tilde{T}$ has a unique fixed point in $\tilde{X}_{\delta}$. Lemma 4.1 has been shown.
\end{proof}
%%%%%%%%%%%%%%%%%%%%%%%%%%%%%%%%%%%%%%%%%%%%%%%%%%%%%%%%%%%%%%%%%%%%%%%%%%%%%%%%%%%%%%%%%%%%%%%%%%%%%%%%%%%%%%%%%%
%%%%%%%%%%%%%%%%%%%%%%%%%%%%%%%%%%%%%%%%%%%%%%%%%%%%%%%%%%%%%%%%%%%%%%%%%%%%%%%%%%%%%%%%
%%%%%%%%%%%%%%%%%%%%%%%%%%%%%%%%%%%%%%%%%%%%%%%%%%%%%%%%%%%%%%%%%%%%%%%%%%%%%%%%%%%%%%%%%%%%%%%%%%%%%%%%%%%%%%%%%%%%%%%%%%%%%%%%%%%%%%%%%%%%%%%%%%%%%%%%%%%%%%%
\section{Proof of Theorem 1.3}
In Section 5, we will show Theorem 1.3. Since $a(x,z)$ is holomorphic at $z=0$ by (ii) of Assumption 1.1, it is sufficient to show that 
\begin{equation}
\partial_{z}^{l}a_1(x,0)=\partial_{z}^{l}a_2(x,0), \ x \in \mathbb{R}^d,\label{5.0}
\end{equation}
for all $l \in \mathbb{N}$. Let $N \in \mathbb{N}$ and let $g_j \in L^2(\mathbb{S}^{d-1})$ ($j=1,2,...,N+1$). Let $\delta \in \left(0, \mathrm{min}(\delta_0, \tilde{\delta}_0) \right)$ be chosen as sufficiently small and depending on $N$ and $g_j$. ($\delta_0, \tilde{\delta}_0$ are corresponding to Theorem 1.2 and Lemma 4.1, respectively.) From Section 4, we obtain the unique solution $r_{\epsilon,j} \in \tilde{X}_{\delta}$ ($j=1,2$) such that
\begin{equation}
\Delta r_{\epsilon,j}+a_j(x, r_{\epsilon,j}+v_{\epsilon})+k^2r_{\epsilon,j}=0 \ \mathrm{in} \ \mathbb{R}^d, \label{5.1}
\end{equation}
where $r_{\epsilon,j}$ satisfies the Sommerfeld radiation, and $v_{\epsilon}$ is given by (\ref{4.1}). The solution $r_{\epsilon,j}$ has the form
\begin{equation}
r_{\epsilon,j}(x)=\int_{\mathbb{R}^d}\Phi(x,y)a_j(y, r_{\epsilon,j}(y)+v_{\epsilon}(y))dy, \ x \in \mathbb{R}^{d}, \ \epsilon \in (0,\delta)^{N+1}.\label{5.2}
\end{equation}
By the assumption of Theorem 1.3 we have 
\begin{equation}
r^{\infty}_{\epsilon,1}(\hat{x})=r^{\infty}_{\epsilon,2}(\hat{x}), \ \hat{x} \in \mathbb{S}^{d-1}, \ \epsilon \in (0,\delta)^{N+1},\label{5.3}
\end{equation}
where $r^{\infty}_{\epsilon,j}$ is a scattering amplitude for $r_{\epsilon,j}$, and it has the form
\begin{equation}
r^{\infty}_{\epsilon,j}(\hat{x})=\int_{\mathbb{R}^d}e^{-ik\hat{x}\cdot y}a_j(y, r_{\epsilon,j}(y)+v_{\epsilon}(y))dy, \ \hat{x} \in \mathbb{S}^{d-1}, \ \epsilon \in (0,\delta)^{N+1}.\label{5.4}
\end{equation}
In order to linearize (\ref{5.2}), we will differentiate it with respect to $\epsilon_l$ ($l=1,...,N+1$), which is possible because $r_{\epsilon,j} \in \tilde{X}_{\delta}$. Then, we have
\begin{equation}
\partial_{\epsilon_l}r_{\epsilon,j}(x)=\int_{\mathbb{R}^d}\Phi(x,y)\partial_{z}a_j(y, r_{\epsilon,j}(y)+v_{\epsilon}(y))(\partial_{\epsilon_l}r_{\epsilon,j}(y)+\delta^2 v_{g_l}(y))dy.\label{5.5}
\end{equation}
As $\epsilon \to +0$ we have by setting $q_j:=\partial_{z}a_j(y, 0)$
\begin{equation}
w_{l,j}(x):=\partial_{\epsilon_l}r_{\epsilon,j}\Big|_{\epsilon=0}(x)=\int_{\mathbb{R}^d}\Phi(x,y)q_j(y)(w_{l,j}(y)+\delta^2 v_{g_l}(y))dy, \label{5.6}
\end{equation}
which implies that 
\begin{equation}
\Delta w_{l,j}+k^2w_{l,j}=-q_j(w_{l,j}+\delta^2 v_{g_l}) \ \mathrm{in} \ \mathbb{R}^d. \label{5.7}
\end{equation}
By setting $u_{l,j}:=w_{l,j}+\delta^2 v_{g_l}$ we have
\begin{equation}
\Delta u_{l,j} +k^2 u_{l,j} +q_j u_{l,j}=0 \ \mathrm{in} \ \mathbb{R}^d. \label{5.8}
\end{equation}
By setting $u_l :=u_{l,1}-u_{l,2}(=w_{l,1}-w_{l,2})$ we have 
\begin{equation}
\Delta u_l +k^2 u_l +q_1 u_l= (q_2-q_1)u_{l,2} \ \mathrm{in} \ \mathbb{R}^d, \label{5.9}
\end{equation}
and we also have
\begin{equation}
(q_2 - q_1)u_{h,1}u_{l,2} = u_{h,1}\Delta u_l - u_l \Delta u_{h,1} \ \mathrm{in} \ \mathbb{R}^d. \label{5.10}
\end{equation}
Differentiating (\ref{5.3}) with respect to $\epsilon_l$ and as $\epsilon \to 0$ we have
\begin{equation}
\int_{\mathbb{R}^d}e^{-ik\hat{x}\cdot y}q_1(y)(w_{l,1}(y)+\delta^2 v_{g_l}(y))dy=\int_{\mathbb{R}^d}e^{-ik\hat{x}\cdot y}q_2(y)(w_{l,2}(y)+\delta^2 v_{g_l}(y))dy, \label{5.11}
\end{equation}
which means that $w_{l,1}^{\infty}=w_{l,2}^{\infty}$, where $w_{l,j}^{\infty}$ is a scattering amplitude of $w_{l,j}$. By setting $\hat{w}_l:=w_{l,1}-w_{l,2}$ we have
\begin{equation}
\Delta \hat{w}_l +k^2 \hat{w}_l = 0 \ \mathrm{in} \ \mathbb{R}\setminus \overline{B_R}, \label{5.12}
\end{equation}
where $\hat{w}_l$ satisfies the Sommerfeld radiation condition, and the scattering amplitude $\hat{w}_l^{\infty}$ of $\hat{w}_l$ vanishes. Then, we have $\hat{w}_l=0$ (that is, $u_l=0$) in $\mathbb{R}\setminus \overline{B_R}$, which implies that by the Green's second theorem we have ($l,h=1,...,N+1$)
\begin{eqnarray}
0&=&\int_{\partial B_{R+1}}u_{h,1}\partial_{\nu}u_l - u_l\partial_{\nu}u_{h,1} ds 
\nonumber\\
&=& \int_{B_{R+1}}u_{h,1}\Delta u_l - u_l \Delta u_{h,1}dx
\nonumber\\
&=&
\int_{B_{R}}(q_2 - q_1)u_{h,1} u_{l,2}dx. \label{5.13}
\end{eqnarray}
By (\ref{5.7}), and definition of $H$ and $T_{q_j}$ in Section 2, $u_{l,j}$ can be of the form
\begin{equation}
u_{l,j}=\delta^2 T_{q_j}Hg_l, \label{5.14}
\end{equation}
and dividing by $\delta^4>0$,
\begin{equation}
0=\int_{B_R}(q_2-q_1) T_{q_1}Hg_h T_{q_2}Hg_ldx. \label{5.15}
\end{equation}
Combining Lemma 2.1 with Lemma 2.2, we conclude that $q_1=q_2$.
\par
By induction, we will show (\ref{5.0}). In the first part of this section, the case of $l=1$ has been shown. We assume that 
\begin{equation}
\partial_{z}^{l}a_1(x,0)=\partial_{z}^{l}a_2(x,0), \label{5.16}
\end{equation}
for all $l=1,2,...,N$. We will show the case of $l=N+1$. We alredy have shown that $q_1=q_2$ and $w_{l,1}^{\infty}=w_{l,2}^{\infty}$, which implies that by the uniqueness of the linear Schr\"{o}dinger equation (\ref{5.7}) we have
\begin{equation}
w_{l,1}=w_{l,2} \ \mathrm{in} \ \mathbb{R}^d, \label{5.17}
\end{equation}
for all $l=1,...,N+1$.
\par
We set $q:=q_1=q_2$ and $w_l:=w_{l,1}=w_{l,2}$. By subinduction we will show that for all $h \in \mathbb{N}$ with $1 \leq h \leq N$
\begin{equation}
\partial^h_{\epsilon_{l_1}...\epsilon_{l_h}}r_{\epsilon,1}\Big|_{\epsilon=0}=\partial^h_{\epsilon_{l_1}...\epsilon_{l_h}}r_{\epsilon,2}\Big|_{\epsilon=0}, \label{5.18}
\end{equation}
where $l_1,...l_h \in \{1,...,N+1\}$. We already have shown that (\ref{5.18}) holds for $h=1$. We assume that (\ref{5.18}) holds for all $h\leq K\leq N-1$. (If $N=1$, this subinduction is skipped.) By differentiating (\ref{5.2}) with respect to $\partial^{K+1}_{\epsilon_{l_1}...\epsilon_{l_{K+1}}}$ we have
\begin{eqnarray}
&&\hspace{-1.3cm} \partial^{K+1}_{\epsilon_{l_1}...\epsilon_{l_{K+1}}}r_{\epsilon,j}(x)=\int_{\mathbb{R}^d}\Phi(x,y) \Biggl\{ \partial^{K+1}_{z} a_j(y, r_{\epsilon,j}(y)+v_{\epsilon}(y))\prod_{h=1}^{K+1}(\partial_{\epsilon_{l_h}}r_{\epsilon,j}(y)+\delta^2 v_{g_{l_h}}(y))
\nonumber\\
&& \ \ \ \ \ \ \ \ \ \ \ \ \ \ 
+\partial_{z} a_j(y, r_{\epsilon,j}(y)+v_{\epsilon}(y))\partial^{K+1}_{\epsilon_{l_1}...\epsilon_{l_{K+1}}}r_{\epsilon,j}(y)+R_{K,j}(y,\epsilon)\Bigg\} dy, \label{5.19}
\end{eqnarray} 
where $R_{K,j}(y,\epsilon)$ is a polynomial of $\partial^{h}_{z} a_j(y, r_{\epsilon,j}(y)+v_{\epsilon}(y))$ and $\partial^{h}_{\epsilon_{l_1}...\epsilon_{l_{h}}}\left(r_{\epsilon,j}(y)+v_{\epsilon}(y) \right)$ for $1\leq h \leq K$. As $\epsilon \to 0$ we have
\begin{eqnarray}
&&\hspace{-1cm}\partial^{K+1}_{\epsilon_{l_1}...\epsilon_{l_{K+1}}}r_{\epsilon,j}\Big|_{\epsilon=0}(x)=\int_{\mathbb{R}^d}\Phi(x,y) \Biggl\{ \partial^{K+1}_{z} a_j(y, 0)\prod_{h=1}^{K+1}(w_{l_h}(y)+\delta^2 v_{g_{l_h}}(y))
\nonumber\\
&& \ \ \ \ \ \ \ \ \ \ \ \ \ \ \ \ \ \ \ \ \ \  
+q(y)\partial^{K+1}_{\epsilon_{l_1}...\epsilon_{l_{K+1}}}r_{\epsilon,j}\Big|_{\epsilon=0}(y)+R_{K,j}(y,0) \Biggr\}dy.\label{5.20} 
\end{eqnarray} 
We set $\tilde{w}_{K+1,j}:=\partial^{K+1}_{\epsilon_{l_1}...\epsilon_{l_{K+1}}}r_{\epsilon,j}\Big|_{\epsilon=0}$ and set $\tilde{w}_{K+1}:=\tilde{w}_{K+1,1}-\tilde{w}_{K+1,2}$. By assumptions of induction and subinduction we have $R_{K,1}(y,0)=R_{K,2}(y,0)$ and $\partial^{K+1}_{z} a_1(\cdot, 0)=\partial^{K+1}_{z} a_2(\cdot, 0)$, which implies that
\begin{equation}
\tilde{w}_{K+1}(x)=\int_{\mathbb{R}^d}\Phi(x,y)q(y)\tilde{w}_{K+1}(y)dy,\label{5.21} 
\end{equation}
which is equivalent to 
\begin{equation}
\Delta \tilde{w}_{K+1}+k^2\tilde{w}_{K+1}+q\tilde{w}_{K+1}=0 \ \mathrm{in} \ \mathbb{R}^d,\label{5.22} 
\end{equation}
where $\tilde{w}_{K+1}$ satisfies Sommerfeld radiation condition. By differentiating (\ref{5.3}) with respect to $\partial^{K+1}_{\epsilon_{l_1}...\epsilon_{l_{K+1}}}$ and as $\epsilon \to 0$ we have
\begin{equation}
\tilde{w}_{K+1,1}^{\infty}=\tilde{w}_{K+1,2}^{\infty},\label{5.23} 
\end{equation}
where $\tilde{w}_{K+1,j}^{\infty}$ is a scattering amplitude of $\tilde{w}_{K+1,j}$. (\ref{5.23}) means that $\tilde{w}_{K+1}^{\infty}=0$, which implies that by Rellich theorem, we conclude that $\tilde{w}_{K+1}=0$ in $\mathbb{R}^d$. (\ref{5.18}) for the case of $K+1$ has been shown, and the claim (\ref{5.18}) holds for all $h=1,...,N$ by subinduction.
\par
By differentiating (\ref{5.2}) with respect to $\partial^{N+1}_{\epsilon_{1}...\epsilon_{K+1}}$, and as $\epsilon \to 0$ (the same argument in (\ref{5.19})--(\ref{5.21})) we have
\begin{equation}
\tilde{w}_{N+1}(x)=\int_{\mathbb{R}^d}\Phi(x,y) \bigg\{ \left( \partial_{z}^{N+1}a_1(x,0)-\partial_{z}^{N+1}a_2(x,0) \right)\prod_{h=1}^{N+1}(w_{h}(y)+\delta^2 v_{g_{h}}(y)) \nonumber
\end{equation}
\vspace{-5mm}
\begin{equation}
\ \ \ \ \ \ \ \ \ \ \ +q(y)\tilde{w}_{N+1}(y) \bigg\}dy.\label{5.24} 
\end{equation}
where $\tilde{w}_{N+1,j}:=\partial^{N+1}_{\epsilon_{1}...\epsilon_{l_{N+1}}}r_{\epsilon,j}\Big|_{\epsilon=0}$ and set $\tilde{w}_{N+1}:=\tilde{w}_{N+1,1}-\tilde{w}_{N+1,2}$. This is equivalent to
\begin{equation}
\Delta \tilde{w}_{N+1}+k^2\tilde{w}_{N+1}+q\tilde{w}_{N+1}=-f\prod_{h=1}^{N+1}\delta^2 T_q Hg_h \ \mathrm{in} \ \mathbb{R}^d,\label{5.25} 
\end{equation}
where $f(x):=\partial_{z}^{N+1}a_1(x,0)-\partial_{z}^{N+1}a_2(x,0)$. By differentiating (\ref{5.3}) with respect to $\partial^{N+1}_{\epsilon_{1}...\epsilon_{K+1}}$ and as $\epsilon \to 0$ (the same argument in (\ref{5.23})) we have 
\begin{equation}
\tilde{w}_{N+1}^{\infty}=0, \label{5.26} 
\end{equation}
where $\tilde{w}_{N+1}^{\infty}$ is a scattering amplitude of $\tilde{w}_{N+1}$. Then, we have $\tilde{w}_{N+1}=0$ in $\mathbb{R}\setminus \overline{B_R}$. 
\par
Let $ \tilde{v} \in L^{2}(B_{R+1})$ be a solution of $\Delta \tilde{v}+k^2\tilde{v}+q\tilde{v}=0$ in $B_{R+1}$. By the Green's second theorem and (\ref{5.25}) we have
\begin{eqnarray}
0&=&\int_{\partial B_{R+1}}\tilde{v}\partial_{\nu}\tilde{w}_{N+1} - \tilde{v}\partial_{\nu}\tilde{w}_{N+1}ds 
\nonumber\\
&=& \int_{B_{R+1}}\tilde{v}\Delta \tilde{w}_{N+1} - \tilde{w}_{N+1} \Delta \tilde{v}dx
\nonumber\\
&=&
\int_{B_{R+1}}-f\prod_{h=1}^{N+1}\delta^2 T_q Hg_h \tilde{v} dx, \label{5.27}
\end{eqnarray}
which implies that dividing by $\delta^2>0$
\begin{equation}
\int_{B_{R+1}}f\prod_{h=1}^{N+1}T_q Hg_h \tilde{v} dx=0.\label{5.28} 
\end{equation}
Let $ v \in L^{2}(B_{R+1})$ be a solution of $\Delta v+k^2v+qv=0$ in $B_{R+1}$. By Lemma 2.1 we can choose $g_{N+1}$ as $g_{N+1,j} \in L^{2}(B_{R+1})$ such that $T_qHg_{N+1,j} \to v $ in $L^{2}(B_R)$ as $j \to \infty$. Then, we have that
\begin{equation}
\int_{B_{R+1}}f\prod_{h=1}^{N}T_q Hg_h v \tilde{v} dx=0.\label{5.29} 
\end{equation}
which implies that by Lemma 2.2 
\begin{equation}
f\prod_{h=1}^{N}T_q Hg_h=0. \label{5.30}
\end{equation}
By Theorem 5.1 of \cite{G. Uhlmann}, we can choose a solution $u_h \in L^{2}(B_{R+1})$ ($h=1,...,N$) of $\Delta u_h + k^2 u_h +q u_h =0$ in $B_{R+1}$, which is of the form
\begin{equation}
u_h(x)=e^{x\cdot p_h}(1+\psi_h(x,p_h)),\label{5.31}
\end{equation}
with $\left\| \psi_h(\cdot, p_h) \right\|_{L^{2}(B_{R+1})} \leq \frac{C}{|p_h|}$ where $C>0$ is a constant, and $p_h=a_h+ib_h$, $a_h,b_h \in \mathbb{R}^d$ such that $|a_h|=|b_h|$ and $a_h \cdot b_h =0$ (which implies that $p_h \cdot p_h =0$), and $a_h\neq a_{h'}$, $b_h\neq b_{h'}$.
\par
Multiplying (\ref{5.30}) by $\overline{f}\prod_{h=1}^{N+1} e^{-x \cdot p_h}$ we have
\begin{equation}
|f|^2\prod_{h=1}^{N}e^{-x \cdot p_h}T_q Hg_h =0,\label{5.32}
\end{equation}
which implies that
\begin{equation}
\int_{B_R} |f|^2 \left( \prod_{h=1}^{N-1}e^{-x \cdot p_h}T_q Hg_h \right) e^{-x \cdot p_N}T_q Hg_Ndx =0.\label{5.33}
\end{equation}
By Lemma 2.1, there exists a sequence $\{g_{N,j} \}_{j\in \mathbb{N}} \subset L^{2}(\mathbb{S}^{d-1})$ such that $T_qHg_{N,j} \to u_N=e^{x\cdot p_N}(1+\psi_N(x,p_N)) $ in $L^{2}(B_R)$ , which implies that
\begin{equation}
\int_{B_R} |f|^2\left( \prod_{h=1}^{N-1}e^{-x \cdot p_h}T_q Hg_h \right) (1+\psi(x,p_N))dx =0.\label{5.34}
\end{equation}
As $|a_N|=|b_N| \to \infty$ in (\ref{5.34}) we have 
\begin{equation}
\int_{B_R}|f|^2\prod_{h=1}^{N-1}e^{-x \cdot p_h}T_q Hg_h =0.\label{5.35}
\end{equation}
Repeating the operation (\ref{5.33})--(\ref{5.35}) $N-1$ times, we have that
\begin{equation}
\int_{B_R}|f|^2 =0,\label{5.36}
\end{equation}
which conclude that $f=0$. By induction, we conclude that (\ref{5.0}) for all $l\in \mathbb{N}$. Therefore, Theorem 1.3 has been shown.
%%%%%%%%%%%%%%%%%%%%%%%%%%%%%%%%%%%%%%%%%%%%%%%%%%%%%%%%%%%%%%%%%%%%%%%%%%%%%%%%%%%%%%%%%%%%%%%%%%%%%%%%%%%%%%%%%%%%%%%%%%%%%%%%%%%%%%%%%%%%%%%%%%%%%%%%%%%%%%%%%%%%%%%%%%%%%%%%%%%%%%%%%%%%%%%%%%%%%%%%%%%%%%%%%%%%%%%%%%%%%%%%%%%%%%%%%%%%%%%%%%%%%%%
%%%%%%%%%%%%%%%%%%%%%%%%%%%%%%%%%%%%%%%%%%%%%%%%%%%%%%%%%%%%%%%%%%%%%%%%%%%%%%%%%%%%%%%%%%%%%%%%%%%%%%%%%%%%%%%%%%%%%%%%%%%%
\section*{Acknowledgments}
The author thanks to Professor Mikko Salo, who supports him in this study, and gives him many comments to improve this paper.

%%%%%%%%%%%%%%%%%%%%%%%%%%%%%%%%%%%%%%%%%%%%%%%%%%%%%%%%

\noindent Graduate School of Mathematics, Nagoya University, Japan.
\\
E-mail address: takashi.furuya0101@gmail.com

\end{document}